\documentclass[11pt,twoside]{article}

\usepackage{mathrsfs,amsfonts,amsmath,amssymb}
\usepackage{latexsym}
\usepackage{comment} 
\newtheorem{satz}{Theorem}

\newtheorem{theorem}[satz]{Theorem}
\newtheorem{lemma}[satz]{Lemma}

\newtheorem{corollary}[satz]{Corollary}
\newtheorem{remark}[satz]{Remark}

\def\Z{\mathbb {Z}}
\def\F{\mathbb {F}}
\def\E{\mathsf{E}}

\def\a{\alpha}

\def\C{\mathbb{C}}

\def\d{\delta}
\def\o{\omega}
\def\({\big (}
\def\){\big )}

\def\G{\Gamma}

\def\le{\leqslant}
\def\ge{\geqslant}
\def\_phi{\varphi}
\def\eps{\varepsilon}

\def\Gr{{\mathbf G}}
\def\FF{\widehat}

\def\D{\Delta}

\def\T{\mathsf{T}}

\def\C{\mathbb{C}}

\def\SL{{\rm SL}}
\def\GL{{\rm GL}}

\def\Cay{{\rm Cay}}

\author{I.D. Shkredov}
\title{On a girth--free variant of the Bourgain--Gamburd machine 
\footnote{This work is supported by the Russian Science Foundation under grant 19--11--00001.}}
\date{}
\begin{document}

	\maketitle


\begin{center}
	Annotation.
\end{center}

{\it \small
    A variant of the Bourgain--Gamburd machine without using any girth bounds is obtained.  
    Also, we find series of applications of 
    the Bourgain--Gamburd machine 
    to problems of Additive Combinatorics, Number Theory and Probability. 
}
\\

\section{Introduction}

    Let $p$ be a prime number and $A \subseteq \SL_2 (\F_p)$ be a set of matrices.
    Suppose for simplicity that $A=A^{-1}$, that is, $A$ is a symmetric set. One can consider the Cayley graph $\Cay (\SL_2(\F_p), A)$ (here $\Cay (\SL_2(\F_p), A) = (\SL_2(\F_p), E)$ and the set of edges $E$ is defined as  $(x,y)\in E$ iff $y=xa$, $a\in A$)  of the set $A$ and study the properties of $\Cay (\SL_2(\F_p), A)$. 
    It is a    
    fundamental 
    problem to show that $A$ is an expander  \cite{Lubotzky123}, \cite{H} under some conditions on $A$. 
    Equivalently, 
    we want to estimate nontrivially the operator norm of all representations of the Fourier transform of the characteristic function of $A$, i.e., $\FF{A}(\rho)$,  $\rho \neq 1$.
    In \cite{BG} (also, see \cite{BGS}) Bourgain and Gamburd obtained 

\begin{theorem}
    Let $A \subseteq \SL_2 (\F_p)$ be a generating set, $A=A^{-1}$, $\tau \in (0,1)$ be a real number  and 
$$
    \mathrm{girth} (\Cay (\SL_2(\F_p), A) \ge \tau \log_{|A|} p \,.
$$
    Then $\Cay (\SL_2(\F_p),A)$ is an expander. 
\label{t:BG_girth}
\end{theorem}

    Recall that the girth of a finite graph is the length of its minimal cycle. 
    The proof of the result above   involves  some calculations with free groups \cite{BG}.  
    Moreover, usually in applications,  
    Theorem \ref{t:BG_girth} 
    is used for a set 
    of generators $A\subseteq \SL_2(\F_p)$ such that 
    $A$ generates a free subgroup of $\SL_2(\Z)$, e.g., 
\begin{equation}\label{f:s,t_free}
A = \left \langle 
\left(\begin{array}{cc}
1 & s \\
0 & 1
\end{array}\right) 
\,, 
\left(\begin{array}{cc}
1 & 0 \\
t & 1
\end{array}\right) \right\rangle \,, \quad \quad |st|\ge 4 
\end{equation} 
    or for a randomly chosen $A$ (it is known that for such $A$ the Cayley graph $\Cay (\SL_2(\F_p), A)$ has large girth), see \cite{BG}.
    To demonstrate transparently the strength of the result above we formulate a consequence of Theorem \ref{t:BG_girth}, which was obtained in \cite{NG_S}  
    (actually, in the only case $g(x)=-1/x$, see the complete proof in Theorem \ref{t:BG'} below) and which was found further applications to the Zaremba conjecture  in  \cite{MMS2}. 
    Recall that $\SL_2(\F_p)$ acts on the projective line via M\"obius transformations: $x\to gx = \frac{ax+b}{cx+d}$, where $g=\left(\begin{array}{cc}
a & b \\
c & d
\end{array}\right)$.

\begin{theorem}
    Let $N\ge 1$ be a sufficiently large integer,  $A,B\subseteq \F_p$ be sets, and $g\in \SL_2 (\F_p)$ be a non--linear map. 
    Then there is 
    an absolute 
    constant $\kappa>0$ such that 
\begin{equation}\label{f:BG'_intr}
    |\{ g(c+a) = c+b ~:~ c \in 2\cdot [N],\,a\in A,\, b\in B \}|
    - \frac{|A||B|N}{p}
    \ll_g \sqrt{|A||B|} N^{1-\kappa} \,.
\end{equation}
    In particular, the Cayley graph $\Cay (\SL_2(\F_p),S)$, where   
    $$
S=\left\{
\left(\begin{array}{cc}
1 & -2j \\
0 & 1
\end{array}\right)
g
\left(\begin{array}{cc}
1 & 2j \\
0 & 1
\end{array}\right) ~:~ 1 \leq j \leq N\right\} \subset \SL_2 (\F_p) \,,
$$
is an expander with $\|\FF{S}(\rho)\|\ll |S|^{1-\kappa}$, 
$\kappa >0$ for all non--trivial  unitary irreducible representations $\rho \neq 1$.
\label{t:BG'_intr}
\end{theorem}
    
    As we said before the  proof of Theorem
    \ref{t:BG'_intr} 
    uses some calculations 
    in 
    free groups (one should take 
    $s=t=2$ in \eqref{f:s,t_free}),  
    as well as some 
    good lower bounds for the girth of the correspondent Cayley graph.
    We avoid 
    to use 
    this technique 
    and obtain a more general (in the sense that the set $S$ below can be much more general than $(c,c)$, $c\in 2\cdot [N]$) and more applicable result.

\begin{theorem}
    Let $\delta \in (0,1]$, 
    $N\ge 1$ be a sufficiently large integer, $N\le p^{c\delta}$ for an absolute constant $c>0$, 
    $A,B\subseteq \F_p$ be sets, and $g\in \SL_2 (\F_p)$ be a non--linear map. 
    Suppose that $S$ is a set, 
    $S \subseteq [N]\times [N]$, $|S| \ge N^{1+\delta}$.
    Then there is 
    a
    constant $\kappa = \kappa (\d) >0$ such that 
\begin{equation}\label{f:BG_new_intr}
    |\{ g(\a+a) = \beta+b ~:~ (\a,\beta) \in S,\,a\in A,\, b\in B \}|
    - \frac{|S||A||B|}{p}
    \ll_g \sqrt{|A||B|} |S|^{1-\kappa} \,.
\end{equation}
\label{t:BG_new_intr}
\end{theorem}

In Sections \ref{sec:applications}, \ref{sec:shifts} we find several applications of Theorems \ref{t:BG'_intr}, \ref{t:BG_new_intr} to problems of Incidence Geometry in $\F_p \times \F_p$ (Theorem \ref{t:BG_new_intr} itself is a result about incidences between hyperbolae from $S$ and Cartesian product $A\times B$)
and Probability.
We obtain a series of new upper bounds for some equations over $\F_p$
(e.g., see the third part of Theorem \ref{t:shifts} below).  
\begin{corollary}
    Let $X\subseteq \F_p$ be a set, $|X|\le 3p/4$, and $g$ be a non--linear map. 
    Then there is an absolute constant $c>0$ such that 
    for all $s\neq 0$ either   one has $|X\cap (X+s)| \le (1-c)|X|$ or 
    $|g(X) \cap (g(X) + s^{-1})| \le (1-c)|X|$. 
\label{c:rr_intr}
\end{corollary}
The main point is the uniformity on $s$ in Corollary \ref{c:rr_intr}. 
Also, we give a new optimal bound for the  mixing time of a  Markov chain, see Theorem \ref{t:g_Markov}. 
Finally, let us remark that Theorem \ref{t:BG_new_intr} implies the main result of \cite{MMS2} concerning Zaremba's conjecture from the theory of continued fractions (in this case one should choose  $S=[N] \times [N]$). 
Our argument of the proof of Theorem \ref{t:BG_new_intr} is similar to the approach of paper  \cite{s_Chevalley}, where a modular form of Zaremba's conjecture was proved, so such connection between the Bourgain--Gamburd machine and continued fractions  is not very surprising but it nevertheless was not widely known.

The last Section of our paper 
concerns some further applications of $\SL_2 (\F_p)$--actions to the question about 
the  intersections of additive shifts of multiplicative subgroups in $\F^*_p = \F_p \setminus \{0\}$. 
Consider the simplest multiplicative subgroup in $\F_p^*$, namely, the set of quadratic residues 
$$
    R = \{ x^2  ~:~ x\in \F^*_p \} \,.  
$$
The set $R$ is a good constructive model for the randomly chosen  subset of $\F_p^*$ (each element of the set is taken with probability $1/2$).  
For example, 
Vinogradov \cite{Vin2} considered the minimal distance $d(p)$ between quadratic residues 
and 
it was conjectured that 
$d(p) \ll p^{\eps}$, where $\eps>0$ is any number.
The first non--trivial results in this direction were  obtained in \cite{Polya}, \cite{Vin1}. 
In \cite{Burgess1} it was proved that $d(p)\ll p^{1/4} (\log p)^{3/2}$ and the best result at the moment is (see \cite{Burgess2})
\begin{equation}\label{f:Burgess}
    d(p) \ll p^{1/4} \log p \,.
\end{equation}
First of all, we show in the Appendix that bound \eqref{f:Burgess}  can be obtained using a combinatorial method,  
which 
differs from the classical approach and its variations, see \cite{Burgess1}
and \cite{BMT}, \cite{IK}, \cite{KSY}.
We need the only consequence of the Weil bound on multiplicative character, namely, that for any different non--zero  shifts $s_1,\dots,s_k \in \F_p$ one has 
\begin{equation}\label{f:R_s}
    |R_{s_1,\dots,s_k}| := |R \cap (R-s_1) \cap \dots \cap (R-s_k)| \le \frac{p}{2^{k+1}} + k\sqrt{p} \,.
\end{equation}
    (a similar result for general multiplicative subgroups is contained in \cite{sv}). 
    In particular, if $k\sim \log p$, then the intersection of the additive shifts  from  \eqref{f:R_s} is $O(\sqrt{p}\log p)$. 
%
%
    Further, applying the Bougain--Gamburd machine and using  $\SL_2(\F_p)$--actions, we break this  square--root barrier for larger number of shifts,  see Section \ref{sec:shifts}.
    More precisely, having an arbitrary  set of shifts $S=\{s_1,\dots, s_k\}$, we want to  expand this set constructively (that is, any random choice is forbidden) in the spirit of paper \cite{BGKS}  and find $\overline{S} = \{s_1,\dots, s_k, t_1,\dots, t_K\}$, $t_j = t_j(S)$
    such that $|R_{s_1,\dots,s_k,t_1,\dots, t_K}| = o(\sqrt{p})$.
    Our 
    simple but crucial 
    observation is that the group 
    $\SL_2(\F_p)$ acts transparently on the sets $R_{s_1,\dots,s_k}$ (more generally, on sets of the form
    $\G_{s_1,\dots,s_k}$ defined as in \eqref{f:R_s} or even on sets $(\a_1 \G +\beta_1) \cap \dots \cap (\a_k \G +\beta_k)$, $\G<\F_p^*$, $\a_j \in \F_p^*$, $\beta_j\in \F_p$), e.g., 
    \begin{equation}\label{f:R_s_inverse}
        R^{-1}_{s_1,\dots,s_k} = R_{s^{-1}_1,\dots,s^{-1}_k} \,, 
    \end{equation}
    provided $s_j \in R$. 
    Hence our additive/multiplicative  problem on size of intersection of additive shifts of $R$ can be treated via the methods of growth in $\SL_2(\F_p)$ as in \cite{BG}, \cite{H}, \cite{s_CDG} etc. 
    The growth in $\SL_2(\F_p)$ is known to be very fast (as, e.g.,  Theorems \ref{t:BG'_intr}, \ref{t:BG_new_intr} show this) and we apply the Bourgain--Gamburd machine to obtain some results in the direction in Section \ref{sec:shifts}.
    Let us formulate a result 
    of this Section. 
    
\begin{theorem}
    Let $p$ be a prime number, $n$ be a positive integer, $\G < \F_p^*$ be a multiplicative subgroup, $\G=-\G$, and $S \subseteq \F_p$ be an arbitrary set.
    Then there is a constructive set $T=T(S)$
    (namely, defined in \eqref{cond:inclusions}, \eqref{f:shifts1_G} below), $|T| \le 2^n |S|+n$ such that $S\subseteq T$ and  
    $$|\G_{T}| \le (1-c)^n |\G_S| \,,$$
    where $c >0$ is an absolute constant, which does not depend on $S$ and $T$. 
\label{t:shifts_intr}
\end{theorem}

The bound  $|T| \le 2^n |S|+n$  is perhaps 
non--optimal (the author believes in $|T| = O(|S|+n)$) and, probably, should take place for an arbitrary set $T$. 
Nevertheless, it seems like that it is the first result of such type.  

\bigskip 

We thank Igor Shparlinski for useful 
remarks.

\section{Definitions}
\label{sec:def}

Let $\Gr$ be a finite group with the identity $1$.
Given two sets $A,B\subset \Gr$, define  the \textit{product set}  of $A$ and $B$ as 
$$AB:=\{ab ~:~ a\in{A},\,b\in{B}\}\,.$$
In a similar way we define the higher product sets, e.g., $A^3$ is $AAA$. 
Let $A^{-1} := \{a^{-1} ~:~ a\in A \}$. 
As usual, having two subsets $A,B$ of a group $\Gr$,  denote by 
\[
\E(A,B) = |\{ (a,a_1,b,b_1) \in A^2 \times B^2 ~:~ a^{-1} b = a^{-1}_1 b_1 \}| 
\]
the {\it common energy} of $A$ and $B$. 
If $A=B$, then we write $\E(A)$ for $\E(A,A)$. 
Clearly, $\E(A,B) = \E(B,A)$ and by the Cauchy--Schwarz inequality 
\begin{equation}\label{f:energy_CS}
\E(A,B) |A^{-1} B| \ge |A|^2 |B|^2 \,.
\end{equation}
To underline the group operation $*$ we write $\E^*(A)$, e.g., $\E^{+}(A)$ or $\E^\times (A)$ for $\Gr = \mathbb{R}$ or $\Gr=\C$, say, considered with the addition $+$ or with the multiplication $\times$. 
We  use representation function notations like $r_{AB} (x)$ or $r_{AB^{-1}} (x)$, which counts the number of ways $x \in \Gr$ can be expressed as a product $ab$ or  $ab^{-1}$ with $a\in A$, $b\in B$, respectively. 
In a similar way, $r_{ABC} (x)$ counts  the number of ways $x \in \Gr$ can be expressed as a product $abc$, where $a\in A$, $b\in B$, $c\in C$ etc.
For example, $|A| = r_{AA^{-1}}(1)$ and  $\E (A,B) = r_{AA^{-1}BB^{-1}}(1) =\sum_{x\in \Gr} r^2_{A^{-1}B} (x)$. 
In this paper we use the same letter to denote a set $A\subseteq \Gr$ and  its characteristic function $A: \Gr \to \{0,1 \}$. 

Let $g\in \Gr$ and let  $A \subseteq \Gr$ be any set. 
Then  put $A^g = g^{-1} A g^{}$ and, similarly, let $x^g := g^{-1} x g^{}$, where  $x\in \Gr$. 
If $H \subseteq \Gr$ is a subgroup, then we 
use the notation 
$H\le \Gr$ and $H<\Gr$ if, in addition, $H\neq \Gr$.
Having a set $A\subseteq \Gr$ we use the symbol $\langle A \rangle$ 
to denote 
the subgroup, generated by $A$.

We write $\F^*_q$ for $\F_q \setminus \{0\}$, where $q=p^s$, $p$ is a prime number.
In the paper we consider the group $\SL_2 (\F_q) \subset \GL_2 (\F_p)$  of matrices 
\[
g=
\left( {\begin{array}{cc}
	a & b \\
	c & d \\
	\end{array} } \right) = (ab|cd) = (a,b|c,d) \,, \quad \quad a,b,c,d\in \F_q \,, \quad \quad \det(g) = ad-bc=1 \,. 
\] 
We need two specific subgroups of $\SL_2 (\F_q)$, namely, 
\[
    \mathbf{B} = \left( {\begin{array}{cc}
	\lambda & s \\
	0 & \lambda^{-1} \\
	\end{array} } \right)\,, \quad \quad 
	\mathbf{U} = \left( {\begin{array}{cc}
	1 & s \\
	0 & 1 \\
	\end{array} } \right)\,,
	\quad \quad \lambda \in \F_q^*,\, s\in \F_q \,.
\]
    By $g^*$ denote the transpose of a matrix $g$.
    Also, let us fix the notation for a unipotent and the Weyl element of $\SL_2 (\F_q)$, namely, 
\[
    u_s := \left( {\begin{array}{cc}
	1 & s \\
	0 & 1 \\
	\end{array} } \right) \in \mathbf{U} \,,~~  s\in \F_q \,, 
	\quad \quad \quad 
	w := \left( {\begin{array}{cc}
	0 & -1 \\
	1 & 0 \\
	\end{array} } \right) \,.
\]
    Having a matrix $g = (ab|cd) \in \GL_2 (\C)$, we write $\| g\|$ for $\max\{|a|,|b|,|c|,|d| \}$.

The signs $\ll$ and $\gg$ are the usual Vinogradov symbols. 
When the constants in the signs  depend on a parameter $M$, we write $\ll_M$ and $\gg_M$.
Let us denote by $[n]$ the set $\{1,2,\dots, n\}$.
All logarithms are to base $2$.
If we have a set $A$, then we will write $a \lesssim b$ or $b \gtrsim a$ if $a = O(b \cdot \log^c |A|)$, $c>0$.

\section{On the Bourgain--Gamburd machine}
\label{sec:BG_new}

In this Section we obtain Theorems \ref{t:BG'_intr}, \ref{t:BG_new_intr} from the Introduction. We start with a consequence of the ping--pong lemma (e.g., see \cite{MKS}) applied to some $\SL_2$--actions  on $\C$. 
The second part of Lemma \ref{l:ping-pong_g} corresponds to the uniqueness of the continued fraction expansion. 

\begin{lemma}
    Let $s,t\in \Z[i]$
    be some numbers.
    Suppose that $|s|, |t| \ge 2$.
    Then there is no non--trivial words of the form $\lambda I$, $\lambda \in \C$ with the letters $u^*_s$ and $u_t$. 
    In particular, the subgroup of $\SL_2 (\Z[i])$ generated by $u^*_s$ and $u_t$ is free. 
    
    If we have any matrix $z =(ab|cd)$ from $\langle u^*_s, u_t \rangle$,  then we can reconstruct $z$ via $(b,d)$. 
\label{l:ping-pong_g}
\end{lemma}
\begin{proof}
    Put $g = u^*_s$ and $h=u_t$.
    We will use the ping--pong lemma, e.g., see \cite{MKS}.
    Let $A = \{ (x,y)\in \mathbb{C}^2 ~:~ |y| > |x| \}$ and $B = \{ (x,y)\in \mathbb{C}^2 ~:~ |x| > |y| \}$, $A\cap B = \emptyset$. Notice that both regions are invariant under multiplication by any non--zero number $\lambda$: $(x,y) \to \lambda (x,y)$. 
    Then for $(x,y)\in B$ and any $n\in \Z\setminus \{0\}$ one has $g^n (x,y) = (x,snx+y)$ and since 
    $$
        |snx+y|\ge |s| |n| |x| - |y|> (|s||n|- 1)|x| \ge |x| \,,
    $$
    it follows that $g^n (x,y) \in A$. 
    Similarly, taking $(x,y)\in A$ and an arbitrary $m\in \Z\setminus \{0\}$, we derive  $h^m (x,y) = (x + tm y, y)$ and thus  $h^m (x,y) \in B$.
    Having a non--trivial word $g^{n_1} h^{m_1} \dots g^{n_k}  = \lambda I$ in our alphabet $\{ g,h \}$,  we obtain for any $b\in B$
\[
    B \ni \lambda \cdot I b = g^{n_1} h^{m_1} \dots g^{n_k} b \in A
\]
    and this is a contradiction. 
    If $g^{n_1} h^{m_1} \dots g^{n_k} h^{m_k} = \lambda I$ (and,  similarly, $h^{n_1} g^{m_1} \dots h^{n_k} g^{n_k} = \lambda I$), then conjugating by an $g^{n}$, $n\neq n_1,0$, we obtain
    $g^{n_1-n} h^{m_1} \dots g^{n_k} h^{n_k} g^{n}= \lambda I$ and it contradicts the previous calculations. 

    Let $z \in \langle g, h \rangle$ be a non--trivial word in our alphabet $\{ g,h \}$, $z =(ab|cd)$. 
    Suppose that $z= h^{m_1} \dots g^{n_k} h^{m_k}$
    and hence $(b,d)^* = h^{m_1} \dots g^{n_k} (m_k,1)^*$.
    If we have $(b,d)^* = h^{m'_1} \dots g^{n'_l} (m'_l,1)^*$ for other numbers $m'_1,\dots, n'_l, m'_l$, then 
\[
    g^{-n'_l} \dots h^{m_1-m'_1} \dots g^{n_k} (m_k,1)^* = (m'_l,1)^*
\]
    and this is a contradiction because the right--hand side belongs to $B$ but the  left--hand side (if it is non--trivial) belongs to $A$. 
    If $z = h^{m'_1} \dots g^{m'_{l-1}} h^{m'_{l-1}} g^{n'_l}$, then 
    $(b,d)^* = h^{m'_1} \dots g^{m'_{l-1}} (m'_{l-1},1)^*$ and we use the same argument.   
    Other words $z$ can be considered similarly.
This completes the proof.
$\hfill\Box$
\end{proof}

\bigskip 

Now we obtain a generalization of Theorem \ref{t:BG'_intr} from the Introduction.

\begin{theorem}
    Let $N\ge 1$ be a sufficiently large integer,  $A,B\subseteq \F_p$ be sets, and $g\in \SL_2 (\F_p) \setminus \mathbf{B}$. 
    Then there is 
    an absolute 
    constant $\kappa>0$ such that 
\begin{equation}\label{f:BG'}
    |\{ g(c+a) = c+b ~:~ c \in [N],\,a\in A,\, b\in B \}|
    - \frac{|A||B|N}{p}
    \ll_g \sqrt{|A||B|} N^{1-\kappa} \,.
\end{equation}
\label{t:BG'}
\end{theorem}
\begin{proof} 
    We use slightly more general arguments to use it in the proofs of the results below. 
    Let us remind that for $t\in \F_p$ we  write $u_t = (1t|01)$.
    By the Bruhat decomposition and the condition $g\notin \mathbf{B}$ we can write $g=u_{t_1} wd u_{t_2}$, 
    $d=(\lambda 0|0\lambda^{-1})$ and $w=(0(-1)|10)$.
    Let $S=\{ (c,c) ~:~ c\in J\} \subseteq [N]^2$. 
    In this terms our equation from \eqref{f:BG'} can be written as (let $(\a,\beta) \in S$)
\begin{equation}\label{f:prod_form-'}
         h_{\a,\beta}\, a :=  u_{t_1-\beta} w d u_{\a+t_2}\, a = u_{-\beta} g u_\a\, a = b \,.
\end{equation} 
    We split the set of all pairs $(\a,\beta) \in S$ onto congruence classes modulo two. 
    Thus $S$ is a disjoint union of at most $4$ sets $S_{ij}$, $i,j \in \{0,1\}$ and it is sufficient to obtain \eqref{f:BG'} for each set $S_{ij}$.
    With some abuse of the notation we use the same letter $S$ for $S_{ij}$. 
    Let $H = \{h_{\a,\beta} ~:~ (\a,\beta)\in S \} \subseteq \SL_2 (\F_p)$.
    It is easy to check that 
    \begin{equation}\label{tmp:h_h'-}
        h_{\a,\beta} h^{-1}_{\a',\beta'} 
        =
         u_{t_1-\beta} w d u_{\a-\a'} d^{-1} w^{-1} u_{\beta'-t_1}
        =
        u_{t_1-\beta} u^*_{\lambda^2(\a'-\a)} u_{\beta'-t_1} \in H^{} H^{-1} \,, 
    \end{equation}
    where $M^*$ is the transpose of a matrix $M$.
    For an arbitrary positive integer $k$ any element of the set $(H^{}H^{-1})^k$ has the form 
\begin{equation}\label{f:prod_form-}
    u_{t_1-\beta_1} u^*_{\lambda^2(\a'_1-\a_1)} u_{\beta'_1-\beta_2} u^*_{\lambda^2(\a'_2-\a_2)} \dots  
    u_{\beta'_{k-1}-\beta_k} u^*_{\lambda^2(\a'_k-\a_k)} u_{\beta'_k-t_1} \,.
\end{equation}
    Notice that one can easily remove left and right terms  $u_{t_1}$, $u_{-t_2}$ in \eqref{f:prod_form-'}, \eqref{f:prod_form-} redefining $A \to u_{-t_1}A$ and $B\to u_{-t_1}B$. 
    After that it remains to say that the products in \eqref{f:prod_form-'}, \eqref{f:prod_form-} coincide (up to $\lambda^2$, $\lambda = \lambda (g)$) with the products without $d$ and $t_1$, $t_2$, that is, with the case $g(x)=-1/x$.
    Hence one can apply the arguments of Bourgain--Gamburd, see \cite{BG} or \cite[Lemma 4]{MMS2}. 
This completes the proof.
$\hfill\Box$
\end{proof}

\begin{remark}
    From formula \eqref{tmp:h_h'-}, it follows that the dependence on $g$ in \eqref{f:BG'} is, actually, on $\lambda = \lambda (g)$, where $g=u_{t_1} wd u_{t_2}$ and  $d=(\lambda 0|0\lambda^{-1})$ or, in other words, on the lower left corner of $g$. 
\label{r:lambda} 
\end{remark}


Now let us prove the main result of this Section. 

\begin{theorem}
    Let $k$ be a positive integer, $\delta \in (0,1]$, $\delta_* \in (0,1)$,
    $N\ge 1$ be a sufficiently large integer, $N\le p^{c\delta/k}$ for an absolute constant $c>0$, 
    $A,B\subseteq \F_p$ be sets and $g_1,\dots, g_k \in \SL_2 (\F_p) \setminus \mathbf{B}$ be maps. 
    Suppose that $S$ is a set, 
    $S \subseteq [N]^{k+1}$, $|S| \ge N^{k(1+\delta)}$ and the intersection of $S$ with 
    any hyperplane of the form $z_j = const$, $j\in [k+1]$ is at most 
    $|S|^{\delta_*}$.
    Then there is 
    a
    constant $\kappa = \kappa (\d,\d_*) >0$ such that 
\[
    |\{ u_{\a_k} g_k u_{\a_{k-1}} \dots g_2 u_{\a_1} g_1 u_{\a_0} a = b ~:~ (\a_0,\dots, \a_k) \in S,\,a\in A,\, b\in B \}|
    - \frac{|S||A||B|}{p}
\]
\begin{equation}\label{f:BG_new-}
    \ll_{g_1,\dots,g_k} \sqrt{|A||B|} |S|^{1-\kappa} \,.
\end{equation}    
    In particular, for any $g\in \SL_2 (\F_p) \setminus \mathbf{B}$ one has 
\begin{equation}\label{f:BG_new}
    |\{ g(\a+a) = \beta+b ~:~ (\a,\beta) \in S,\,a\in A,\, b\in B \}|
    - \frac{|S||A||B|}{p}
    \ll_g \sqrt{|A||B|} |S|^{1-\kappa} \,.
\end{equation}
\label{t:BG_new}
\end{theorem}
\begin{proof} 
    We use the same notation as in the proof of Theorem \ref{t:BG'} and let us begin with the case $k=1$, which corresponds to formula \eqref{f:BG_new}. 
    As before, we know that for an arbitrary positive integer $m$ any element of the set $(H^{}H^{-1})^m$ has the form 
\begin{equation}\label{f:prod_form}
    u_{-\beta_1} u^*_{\a'_1-\a_1} u_{\beta'_1-\beta_2} u^*_{\a'_2-\a_2} \dots  
    u_{\beta'_{m-1}-\beta_m} u^*_{\a'_m-\a_m} u_{\beta'_m} 
\end{equation}
    (for simplicity we consider just the case $\lambda (g) =1$, the general case is similar because we allow the dependence on $g$ in \eqref{f:BG_new}). 
    Here we have removed $u_{t_1}, u_{-t_1}$ redefining the sets $A$ and $B$ as in the proof of Theorem \ref{t:BG'}.
    Also, we have considered splitting of the set $S$ modulo two. 
    In particular, we see that 
    if $y_1, \dots, y_m \in H H^{-1}$, then  $\| y_1 \dots y_m \| \le (2N)^{2m+1}$.
    Take a positive integer $l$ such that $(2N)^{2l+1} < p$.
    Then  the set of matrices $(HH^{-1})^l$ is, actually, belongs to $\SL_2 (\Z)$. 
    Our task is show that  for any $z\in \SL_2 (\F_p)$ and an arbitrary $\G < \SL_2 (\F_p)$ the following holds 
\begin{equation}\label{f:K_subgr}
    \sum_{x\in z\G} r_{(HH^{-1})^l} (x) \le \frac{|H|^{2l}}{K} \,,
\end{equation}
    where 
\[
    K= N^{-o_s(1)} \min\left\{ 
        \left( \frac{|H|}{N} \right)^{2s} N^{-1}, \,
            |H|^{2s(1-\delta_*)} N^{-1}, \,
            |H|^{2s(1-\delta_*)}
        \right\} 
\]
and $s$ is the maximal integer such that 
    $(2N)^{2s+1} < 2^{-5} p^{1/4}$.
    By our condition we know that the intersection of $S$ with an arbitrary vertical/horizontal line is at most $|S|^{\delta_*}$.
    The  
    definition of the quantity $\d_*$
    implies  the following simple bound 
    $\delta_* \le (1+\delta)^{-1} \le 1-\delta/2$.
    Hence using the assumption $N\le p^{c\delta}$ 
    and taking $s$ sufficiently large, we obtain 
\begin{equation}\label{tmp:06.09_1}
    \log K \gg \min\{ (2s\d - 2), s(1-\delta_*) \} \cdot  \log N \gg \d \log p \,.
\end{equation}
    After that estimate \eqref{f:BG_new} follows by the usual method, see \cite{BG}, \cite{NG_S} or \cite[Lemma 4, formulae (13)--(15)]{MMS2}.
    More precisely, by the H\"older inequality  the error term in \eqref{f:BG_new} can be estimated as 
\begin{equation}\label{tmp:06.09_1*}
    \sqrt{|A||B|} \cdot \left( |B|^{-1} \sum_s r_{(HH^{-1})^l} (s) \sum_{x\in B} B(sx) \right)^{1/2l} 
    \ll 
       \sqrt{|A||B|} |S| p^{-\zeta/2l}  \,,
\end{equation} 
    where $\zeta = 1/2^{t+2}$ and $t\ll \log p/\log K \ll \d^{-1}$. 
    Recalling the definition of $l$ and using \eqref{tmp:06.09_1}, we get 
$$
    |\{ g(\a+a) = \beta+b ~:~ (\a,\beta) \in S,\,a\in A,\, b\in B \}|
    - \frac{|S||A||B|}{p}
    \ll 
    \sqrt{|A||B|} |S| N^{-\Omega(\zeta)} 
    \ll 
$$
\begin{equation}\label{tmp:06.09_1**}
    \ll 
    \sqrt{|A||B|} |S|  N^{-\exp(-\Omega(1/\delta))} 
\end{equation} 
    as required. 
    
    Thus 
    it remains
    to obtain \eqref{f:K_subgr} and this is the main technical part of our proof. 
    Usually (see \cite{BG}), the simplest case $\G = \{ 1\}$,  which corresponds to the estimate $\| r_{(HH^{-1})^l} \|_\infty \le \frac{|H|^{2l}}{K}$ is considered separately and it is useful and instructive to follow this classical scheme. 
    By the ping--pong Lemma \ref{l:ping-pong_g} 
    we see that two elements of form \eqref{f:prod_form}
\[
    u_{-\beta_1} u^*_{\a'_1-\a_1} u_{\beta'_1-\beta_2} u^*_{\a'_2-\a_2} \dots  
    u_{\beta'_{l-1}-\beta_l} u^*_{\a'_l-\a_l} u_{\beta'_l} 
    =
\]
\begin{equation}\label{f:products_eq}
    =
    u_{-\tilde{\beta}_1} u^*_{\tilde{\a}'_1-\tilde{\a}_1} u_{\tilde{\beta}'_1-\tilde{\beta}_2} u^*_{\tilde{\a}'_2-\tilde{\a}_2} \dots  
    u_{\tilde{\beta}'_{l-1}-\tilde{\beta}_l} u^*_{\tilde{\a}'_l-\tilde{\a}_l} u_{\tilde{\beta}'_l} 
\end{equation}
    coincide (recall that we work modulo two, that is, all variables are even, say) if and only if $\beta_1 = \tilde{\beta}_1$, $\beta'_l = \tilde{\beta}'_l$ 
    and $\a_j-\a'_j = \tilde{\a}_j-\tilde{\a}'_j$, $j\in [l]$, $\beta'_j - \beta_{j+1} = \tilde{\beta}_j - \tilde{\beta}'_{j+1}$, $j\in [l-1]$.
    Fixing $\tilde{\a}_j$, $\tilde{\beta}_j$, $j\in [l]$ and using 
    the definition of the quantity $\delta_*$, 
    we see that  the number of possible pairs $(\a_j,\beta_j)$ does not exceed $|S|^{(2l-1)\delta_*}$
    and hence $\| r_{(HH^{-1})^l} \|_\infty \le |S|^{(2l-1)\delta_*} \le |S|^{2l\delta_*} \le \frac{|H|^{2l}}{K}$ as required.

    Below we use the argument similar to paper \cite{s_Chevalley}. 
    First of all, consider the case when  
    $\G$ is a Borel subgroup. 
    Take $h= (p_{s-1} p_s | q_{s-1} q_s) \in (HH^{-1})^s$, where $(2N)^{2s+1} < 2^{-5} p^{1/4}$ and consider the inclusion 
\begin{equation}\label{f:inclusion_gen_Borel}
g_1 h g_2 = 
\left( {\begin{array}{cc}
	\alpha & \beta \\
	\gamma & \delta \\
	\end{array} } \right) 
\left( {\begin{array}{cc}
	p_{s-1} & p_s \\
	q_{s-1} & q_s \\
	\end{array} } \right) 
\left( {\begin{array}{cc}
	a & b \\
	c & d \\
	\end{array} } \right) 
\in \mathbf{B}
\end{equation}
    In other words, $h\in g^{-1}_1 \mathbf{B} g^{-1}_2$ and hence taking another $h' = (p'_{s-1} p'_s | q'_{s-1} q'_s) \in (HH^{-1})^s$, we have 
    $h (h')^{-1} \in g^{-1}_1 \mathbf{B} g_1$. 
    Suppose that $g_1\in \mathbf{B}$ (the case $g_2\in \mathbf{B}$ can be considered similarly). 
    Then it is easy to see that $q'_s q_{s-1} \equiv q_s q'_{s-1} \pmod p$ hence $q'_s q_{s-1} = q_s q'_{s-1}$ and thus $q_s=q'_s$, $q_{s-1}=q'_{s-1}$.
    In other words, 
    the pair $(q_{s-1},q_s)$ is determined uniquely.
    Writing $\frac{p_s}{q_s} = [b_1,\dots,b_s]$,  we 
    obtain 
    that $\frac{q_s}{q_{s-1}} = [b_s,\dots,b_1]$ 
    and we can reconstruct the matrix $h$ (see details in \cite{s_Chevalley} or just use the second part of our ping--pong Lemma \ref{l:ping-pong_g}).
    Whence 
    we 
    get 
    as in formula \eqref{f:products_eq} that the number of possible inclusions is at most $|S|^{(2s-1) \delta_*}$. 
    Similarly, $(h')^{-1}h \in g_2 \mathbf{B} g_2^{-1}$ and if $g_2 \in \mathbf{B}$, then $p_{s-1}q'_{s-1} = q_{s-1}p'_{s-1}$. 
    Hence $p_{s-1} = p'_{s-1}$, $q_{s-1}=q'_{s-1}$ and we can reconstruct $(p_{s},q_{s})$ from $(p_{s-1},q_{s-1})$ in at most $2N$ ways. 
    Thus the number of possible inclusions is at most $2N |S|^{(2s-3)\delta_*}$.

    Now we can assume that both $g_1,g_2 \notin \mathbf{B}$. In view of the Bruhat decomposition 
	(i.e. one can put $d = \alpha = 0$, $\beta = b =1$, $\gamma = c = -1$) or just a direct calculation, it is easy to see (or consult \cite{s_Chevalley}) that inclusion \eqref{f:inclusion_gen_Borel} is equivalent to 
\begin{equation}\label{f:A_gBh''}
	\delta (q_s + \omega q_{s-1}) \equiv p_s + \omega p_{s-1} \pmod p \,,
\end{equation}	
	where $\omega = -a$. 
	Equation \eqref{f:A_gBh''} can be interpreted easily: any Borel subgroup fixes a point (the standard Borel subgroup fixes $\infty$) and hence inclusion \eqref{f:inclusion_gen_Borel} says that  our set $(HH^{-1})^s$ transfers $\omega$ to $\delta$.
	In other terms, identity \eqref{f:A_gBh''} says that the tuples $(q_s,q_{s-1},p_s, p_{s-1})$ belongs to a hyperspace with the normal vector 
	$(\delta,\delta \omega,-1,-\omega)$ and hence for some other solutions $(q'_s,q'_{s-1},p'_s, p'_{s-1})$, $(q''_s,q''_{s-1},p''_s, p''_{s-1})$, $(q'''_s,q'''_{s-1},p'''_s, p'''_{s-1})$ of \eqref{f:A_gBh''}, we get 
\begin{eqnarray}\label{f:det_A}	
	\begin{vmatrix}
		q_s & q_{s-1} & p_s & p_{s-1} \\ 
		q'_s & q'_{s-1} & p'_s & p'_{s-1} \\ 
		q''_s & q''_{s-1} & p''_s & p''_{s-1} \\ 
		q'''_s & q'''_{s-1} & p'''_s & p'''_{s-1} 
	\end{vmatrix}
	 \equiv 0 \pmod p \,.
\end{eqnarray}
    	If we solve equation \eqref{f:det_A} with elements from $(HH^{-1})^s$, then we arrive to an equation 
\begin{equation}\label{tmp:17.10_1'}
	X q_s + Y q_{s-1} + Z p_s + W p_{s-1} \equiv 0 \pmod p \,,
\end{equation}
	where $|X|, |Y|, |Z|, |W| < 2^{-2} p^{3/4}$, which is, actually, an equation in $\Z$. 
	We can assume that not all integer coefficients $X,Y,Z,W$ (which itself are some determinants of the matrix from \eqref{f:det_A}) vanish because otherwise we obtain a similar equation with a smaller number of variables.
	Combining \eqref{tmp:17.10_1'} and the identity $q_s p_{s-1} - p_s q_{s-1} = 1$, we derive
\[
	q_{s-1} p_s X = -p_{s-1} (Y q_{s-1} + Z p_s + W p_{s-1}) -X 
\]
	or, in other words, 
\begin{equation}\label{tmp:divisors}
	(X q_{s-1} + Z p_{s-1})(X p_s + Y p_{s-1}) = YZ p^2_{s-1} - X(W p_{s-1}^2 + 1) := f(p_{s-1}) \,.
\end{equation}
	Fix $p_{s-1} \le (2N)^{2s+1} < 2^{-5} p^{1/4}$ and suppose that $f(p_{s-1}) \neq 0$. 
	Then the number of the solutions to equation \eqref{tmp:divisors} can be estimated in terms of the divisor function as $N^{o_s(1)}$. 
	Further if we know $(q_{s-1},p_{s}, p_{s-1})$, then in view of \eqref{tmp:17.10_1'} we determine our matrix from $(HH^{-1})^s$ 
	uniquely. 
	Now in the case $f(p_{s-1}) = 0$, we see that there are at most two choices  for $p_{s-1}$ and fixing $q_s \le (2N)^{2s+1}$ 	we find the remaining variables using formulae  \eqref{tmp:17.10_1'}, \eqref{tmp:divisors}.
	Thus in view of our condition $|S| \ge N^{1+\d}$ 
    we obtain 
\[
    \max_{g_1,g_2\in \SL_2 (\F_p)}\, \sum_{x\in g_1 \mathbf{B} g_2}  r_{(HH^{-1})^l} (x)
    \le |H|^{2l-2s} \cdot \max_{g_1,g_2\in \SL_2 (\F_p)}\, \sum_{x\in g_1 \mathbf{B} g_2}  r_{(HH^{-1})^s} (x)
    \le 
\]
\begin{equation}\label{tmp:03.09_1}
    \le
    |H|^{2l-2s} 
    N^{o_s(1)} 
    \left( (2N)^{2s+1} + 2N |H|^{\delta_* (2s-3)} + |H|^{\delta_* (2s-1)}\right) \,.
\end{equation}
    and \eqref{f:K_subgr} follows in the case of any Borel subgroup $\G$.

    It remains to consider the rest of  maximal subgroups from $\SL_2 (\F_p)$ but the structure of the lattice of the subgroups is known to be very simple for this group, see \cite{Dickson}. Excluding subgroups  of finite size (and considered Borel subgroups) any maximal subgroup is a dihedral  group of size $O(p)$, see \cite[Theorems 6.17, 6.25]{Dickson}, \cite{BG} and \cite{NG_S}. 
    Below we use rather rough arguments just to show that an analogue of bound \eqref{tmp:03.09_1} takes place for any dihedral  subgroup, of course it will be enough for our purpose (more delicate calculations can be found in \cite{NG_S}). 
    Thus we consider 
\[
\mathcal{C}_{\varepsilon}:=\left\{\left(\begin{array}{cc}
u & \varepsilon v \\
v & u
\end{array}\right): u, v \in \mathbb{F}_{p},\, u^{2}-\varepsilon v^{2}=1\right\} \,,
\]
    where $\eps$ is a primitive root. 
    Our equation is 
$$
\left(\begin{array}{cc}
x u+y v & \varepsilon x v+y u \\
z u+w v & \varepsilon z v+w u
\end{array}\right)
=
\left(\begin{array}{cc}
x & y \\
z & w
\end{array}\right)\left(\begin{array}{cc}
u & \varepsilon v \\
v & u
\end{array}\right)=\left(\begin{array}{cc}
p_{s-1} & p_{s} \\
q_{s-1} & q_{s}
\end{array}\right)\left(\begin{array}{cc}
X & Y \\
Z & W
\end{array}\right)
=
$$
$$
=
\left(\begin{array}{cc}
p_{s-1}X + p_s Z & p_{s-1}Y + p_s W \\
q_{s-1}X + q_s Z & q_{s-1}Y + q_s W
\end{array}\right)
$$
with $x w-y z=X W-Y Z=1.$ 
It follows that 
$$
    X = q_s (xu + yv) - p_s (zu+wv) = (q_s x - p_s z)u + (q_s y - p_s w) v = A u + B v \,,
$$
and 
$$
    Y = q_s (\eps xv+yu) - p_s (\eps zv + wu) = (q_s y - p_s w)u + (q_s \eps x-p_s \eps z) v = Cu + Dv \,.
$$
    From 
    $xw-yz=1$ one has 
    $(A,B)\neq (0,0)$ and $(C,D)\neq (0,0)$. 
    For concreteness let us  assume that $A\neq 0$, $C\neq 0$.
    Using the last equations, as well as the identity $u^{2} = \varepsilon v^{2} + 1$ and multiplying it by $A^2\neq 0$ and $C^2\neq 0$, correspondingly, we get 
\begin{equation}\label{f:v1}
    \a v^2 + \beta v + \gamma := (B^2 - \eps A^2) v^2 - 2BXv + X^2-A^2 = 0
\end{equation}
    and, similarly, 
\begin{equation}\label{f:v2}
    \a_* v^2 + \beta_* v + \gamma_* := 
    (D^2 - \eps C^2) v^2 - 2DYv + Y^2-C^2 = 0
\end{equation}
    Since $\eps$ is a primitive root and hence in particular, $\eps$ is not a square, it follows that the quadratic equations are non--trivial. In other words, $\a \neq 0$ and $\a_* \neq 0$ for any $(p_s,q_s)$. 
    We can assume that $v\neq 0$ because otherwise it gives at most eight    points in our intersection. 
    Now if $v\neq 0$, then excluding $v$ from \eqref{f:v1}, \eqref{f:v2},  we arrive to the relation between  $p_s$ and $q_s$, namely, 
\[
    (\a \gamma_* - \a_* \gamma)^2 = (\beta \gamma_* - \beta_* \gamma) (\a \beta_* - \beta \a_*) \,.
\]
    One can check that this is a non--trivial equation and hence \eqref{f:K_subgr} follows with  $K= \left( \frac{|H|}{N} \right)^{2s} N^{-1-o_s(1)}$. 
    Indeed, the homogeneous part of degree eight of the last equation is $((BC)^2 - (DA)^2)^2$ and hence it is zero iff
    $\eps (q_{s} x- p_s z)^2 = - (q_s y - p_s w)^2$. 
    It follows that $y^2 + \eps x^2 = w^2+\eps z^2 = 0$ and $-\eps xz = yw$ (otherwise we have a non--trivial equation in $q_s,p_s$). It is easy to check using $xy-zw=1$ that this is impossible. 

    Now 
    it remains to 
    obtain \eqref{f:BG_new-} and we use similar arguments as above. 
    We take $H = \{ u_{\a_k} g_k u_{\a_{k-1}} \dots g_2 u_{\a_1} g_1 u_{\a_0} ~:~ (\a_0,\dots, \a_k) \in S \}$ and derive an analogue of formula  \eqref{f:prod_form} for elements of $(HH^{-1})^m$
\[
    u_{\a_k} g_k u_{\a_{k-1}} \dots g_2 u_{\a_1} g_1 u_{\a_0-\a^{(1)}_0} g^{-1}_1 u_{-\a^{(1)}_1} g^{-1}_2 
    \dots 
\]
\begin{equation}\label{f:prod_g_k}
    \dots 
    g_k^{-1} u_{\a^{(2)}_k - \a^{(1)}_k} \dots u_{\a^{(m-2)}_0-\a^{(m-1)}_0} g^{-1}_1 u_{\a^{(m-1)}_{1}} \dots g^{-1}_{k-1} u_{-\a^{(m-1)}_{k-1}} g^{-1}_k u_{-\a^{(m-1)}_k} \,.
\end{equation}
    As in the proof of Theorem \ref{t:BG'} we can assume that $g_j = u'_{j} w u''_j$ (in other words, loosing the constants,  which depend on $g_1,\dots,g_k$ we can suppose that $\lambda (g_j)$ and $t_1 (g_j)$ in \eqref{f:prod_form-} equal  one).
    Thus as above in formula \eqref{f:products_eq}, we see that $\| r_{(HH^{-1})^m} \|_\infty \le |S|^{\delta_* (2m-1)}$.
    The only difference between the case $k=1$ is a new bound for $q_s$ (again $(p_s,q_s)$ determines the matrix uniquely thanks to the second part of Lemma \ref{l:ping-pong_g} or, alternatively, via the uniqueness of the continued fraction expansion).
    From \eqref{f:prod_g_k} we see that any element of   $(HH^{-1})^m$ does not exceed $(2N)^{2km+1}$ and hence now we can define $s$
    as the maximal integer 
    such that 
    $(2N)^{2ks+1} < 2^{-5} p^{1/4}$. 
    On the other hand, for any $m$ one has  $\| r_{(HH^{-1})^m} \|_1 = |H|^{2m} \ge N^{2mk(1+\d)}$ and as in \eqref{tmp:06.09_1}, we derive
\[
    K = N^{-o_s(1)} \min \left\{ \left( \frac{|H|}{N^k} \right)^{2s} N^{-1}, 
    |H|^{2s(1-\delta_*)} N^{-1}, 
    |H|^{2s(1-\delta_*)} \right\}
    \ge 
        p^{\Omega(\d)} 
\]
    because the condition $N\le p^{c\delta/k}$ allows us to take $s$ to be sufficiently large such that $s\gg 1/\d$. 
    As before we have used a simple  bound   $\delta_* \le (1+\delta)^{-1} \le 1-\delta/2$, which follows from the definition of the quantity $\delta_*$.  
    The rest of the argument coincides with the case $k=1$. 
    From calculations \eqref{tmp:06.09_1*}---\eqref{tmp:06.09_1**},   it follows that $\kappa$ does not depend on $k$ (just because $K$ does not depend on $k$).  
This completes the proof.
$\hfill\Box$
\end{proof}

\begin{remark}
    The condition $|S|>N^{1+\delta} = |[N]^2|^{1/2+\delta/2}$ can be interpreted as "Hausdorff dimension of the  correspondent Cantor--type set"\, is greater than $1/2$, see \cite{s_Chevalley}. 
    Also, the condition $g_1,\dots, g_k \notin \SL_2 (\F_p)\setminus \mathbf{B}$ is not really crucial. One can see from the proof that the argument works (with a worse constant $\kappa(\d)>0$, of course) if just one $g_j$ is a non--linear map. 
\end{remark}

{\bf Question.}
Is it possible to relax the condition $|S| \ge N^{k(1+\d)}$  in Theorem \ref{t:BG_new} to $|S| \ge N^\delta$ (even in the case $k=1$)?

\bigskip 

Now let us formulate a consequence of Theorem \ref{t:BG_new} for large subsets of $\F_p \times \F_p$, having the following "measurable"\, form. 
For simplicity, we use just the case $k=1$ of Theorem \ref{t:BG_new}.

\begin{corollary}
    Let $\delta, \tilde{\d} \in (0,1]$,  
    $A,B\subseteq \F_p$ be sets and $g\in \SL_2 (\F_p) \setminus \mathbf{B}$ be a map. 
    Suppose that $S \subseteq \F_p \times \F_p$ is a set, having the form 
\begin{equation}\label{cond:BG_new_c}
    S = \left( \bigsqcup_{j\in J} S_j \right) \bigsqcup \Omega \,, 
\end{equation}
    where each 
    $S_j$ belongs to a shift of $[N] \times [N]$, $N\le p^{c\d}$, where $c>0$ is an absolute constant.
    Further let $|S_j| \ge N^{1+\delta}$ for all $j\in J$ and $|\Omega| \le |S|^{1-\tilde{\d}}$. 
    Then there is 
    a
    constant $\kappa = \kappa (\d,\tilde{\d}) >0$ such that 
\begin{equation}\label{f:BG_new_c}
    |\{ g(\a+a) = \beta+b ~:~ (\a,\beta) \in S,\,a\in A,\, b\in B \}|
    - \frac{|S||A||B|}{p}
    \ll \sqrt{|A||B|} |S|^{1-\kappa} \,.
\end{equation}
\label{c:BG_new}
\end{corollary}

Indeed, obviously the set $\Omega$ coins $\min\{|A|,|B|\} |S|^{1-\tilde{\d}}$ into \eqref{f:BG_new_c} and for each $S_j$, $j\in J$ one can apply Theorem \ref{t:BG_new}.

\section{First applications}
\label{sec:applications} 


In this Section we obtain a series applications to Incidence Geometry and Probability. We use Theorems  \ref{t:BG'_intr}, \ref{t:BG'}, 
as well as some calculations from the proof of Theorem \ref{t:BG_new}.

Our first application concerns lazy Markov chains, e.g., see \cite{He}.
Namely, applying  Theorem \ref{t:BG'} and using the same scheme as in \cite{s_CDG}, \cite{He}, we immediately obtain

\begin{theorem}
        Let $p$ be a prime number, $\gamma \in \F_p^*$, and $g\in \SL_2 (\F_p) \setminus \mathbf{B}$.
    Also, let $\eps_{j}$ be the random variables distributed  uniformly on $\{ \gamma^{}, -\gamma\}$. 
    Consider the lazy Markov chain $X_0,X_1,\dots, X_n, \dots$ defined by 
\[
    X_{j+1}=\left\{\begin{array}{ll}
 g\left(X_{j}\right) + \varepsilon_{j+1} & \text { with probability } 1 / 2\,, \\
X_{j} & \text { with probability } 1 / 2 \,.
\end{array}\right.
\]
    Then for any $c>0$ and any $n = c \log p$ one has 
\[
    \| P_n - U\| := \frac{1}{2} \max_{A \subseteq \F^*_p} \left| \mathrm{P} (X_n \in A) - \frac{|A|}{p-1} \right| \le e^{-O(c)} \,.
\]
The same is true for the chain $X_{j+1} = g\left(X_{j}\right) + \varepsilon_{j+1}$, where $\eps_j$ denote the random variables distributed  uniformly on $\{ 0, \gamma, -\gamma\}$. 
\label{t:g_Markov}
\end{theorem}


Now we obtain 
two 
applications to some problems from Incidence Geometry over $\F_p$.
The first one concerns M\"obius transformations and  we need 
\cite[Theorem 3.2]{RW} 
(also, see \cite[Theorem 3]{WW_hyp}).

\begin{theorem}
Let $A \times B$ be a set of points in $\mathbb{F}_{p}^{2}$, and let $T$ be any set of M\"obius transformations, $|T|>|A|$, $|A|\le \sqrt{p}$.  
Then 
$$
I(A \times B, T) \ll 
|A|^{4 / 5}|B|^{3 / 5}|T|^{4 / 5}+|A|^{6 / 5}|B|^{7 / 5}|T|^{1 / 5}+|T| \,.
$$
\label{t:hyp_inc}   
\end{theorem}

Using this result we improve \cite[Corollary 3, part 1]{WW_hyp} for sets having non--trivial upper bound for the additive energy $\E^{+}$.

\begin{theorem}
    Let $A,B,C \subseteq \F_p$.
    Then 
\[
    \left| \left\{ (a,a',b,b',c,c') \in A^2 \times B^2 \times C^2 ~:~ \frac{1}{a+b}+c = \frac{1}{a'+b'}+c' \right\} \right|
    \lesssim 
\]
\begin{equation}\label{f:cont2}
    \lesssim 
    |A|^{7/5} |B|^{8/5} |C|^{6/5} (\E^{+}(C))^{1/5} \,. 
\end{equation}
    In particular, for $|A|\le \sqrt{p}$ one has 
\begin{equation}\label{f:cont2'}
    |(A+A)^{-1}+A| = \left| \left\{ \frac{1}{a+b}+c ~:~ a,b,c \in A \right\} \right| \gtrsim 
    |A|^{6/5} \,.
\end{equation}
\label{t:cont2}
\end{theorem}
\begin{proof} 
    Let 
    $\sigma$ be the number of the solutions to equation \eqref{f:cont2}. In terms of the actions our equation is (we redefine $A,B,C$ to keep the general scheme of the proof) 
\[
    u_c w u_b a = u_{c'} w u_{b'} a' 
\]
    or, in other words, 
\[
    u_{-b'} u^*_{c'-c} u_b a = a' \,.
\]
    It means that 
    $$
        \sigma = \sum_{g} r_{GG^{-1}} (g) \sum_{x\in A} A(gx) \,, 
    $$
    where $G=\{ u_{-b'} u^*_{c'}\}_{b'\in B, c'\in C}$, $|G| = |B||C|$.
    Using the H\"older inequality, we get
\[
    \sigma^5 \le \sum_{g\in \SL_2 (\F_p)} \left| \sum_{x\in A} A(gx) \right|^5 \cdot \E(G) |G|^6 \,. 
\]
    Clearly, $\E(G) = |B|^2 \E^{+}(C)$ (or consult formula \eqref{f:products_eq}) and applying  Theorem \ref{t:hyp_inc},  as well as the fact that any M\"obius  transformation can only have at most $|A|$ incidences with the set $A\times A$, we get 
\[
 \sigma^5 \ll |A|^7 |B|^8 |C|^6 \E^{+}(C) \log |A| 
\]
    as required. 
    Using the trivial bound $\E^{+} (C) \le |C|^3$ and the Cauchy--Schwarz inequality, we obtain \eqref{f:cont2'}.  
This completes the proof.
$\hfill\Box$
\end{proof}

\bigskip

\bigskip

Having a  prime number $p\equiv 3 \pmod 4$, we take $i^2=-1$, $i\in \F_{p^2} \setminus \F_p$ and write $\F_{p^2} = \F_q = \F_p [i]$. Any $z\in \F_q$ can be written as $z=x+iy$ with $x,y\in \F_p$ and we identify $\F_q$ with $\F_p \times \F_p$. 
Theorem \ref{t:Z[i]} below is a result on incidences in $\F_p [i]$. Of course, asymptotic formula \eqref{f:Z[i]_as} gives a non--trivial result for double Kloosterman sums  \cite{s_Kloosterman} and we leave this deduction to the interested reader.

\begin{theorem}
    Let $p\equiv 3 \pmod 4$ be a prime number, $\mathcal{A}, \mathcal{B}, \mathcal{C}, \mathcal{D} \subseteq \F^2_p$ be sets and  $0 \neq (\a,\beta) \in \F_p \times \F_p$. 
    Then the system of the equations 
\begin{equation}\label{f:Z[i]_1}
    (a_1 + b_1) (c_1+d_1) - (a_2+b_2)(c_2+d_2) = \a\,,
        \quad 
    (a_2 + b_2) (c_1+d_1) + (a_1+b_1)(c_2+d_2) = \beta\,,
\end{equation}
where  $(a_1,a_2) \in \mathcal{A}$, $(b_1,b_2) \in \mathcal{B}$, $(c_1,c_2) \in \mathcal{C}$, $(d_1,d_2) \in \mathcal{D}$ has 
\begin{equation}\label{f:Z[i]_as}
    \frac{|\mathcal{A}||\mathcal{B}||\mathcal{C}||\mathcal{D}|}{p^2} + O\left( \sqrt{|\mathcal{A}||\mathcal{C}|} (|\mathcal{B}||\mathcal{D}|)^{1-\delta (\eps)} \right)\,, \quad  
\delta (\eps) > 0 
\end{equation}
solutions, provided 
$|\mathcal{B}||\mathcal{D}| \ge (|\mathcal{A}||\mathcal{C}|)^\eps$. 
In particular, if $|A+B|\le K|A|$, $|B| \ge |A|^\eps$, then the number of the solutions to the system of the equations   
\begin{equation}\label{f:Z[i]_2}
    xy-zw=\a\,, \quad xw+yz = \beta\,, \quad x,y,z,w\in A
\end{equation}
    is at most $\frac{K^4|A|^4}{p^2} + O(|A|^2 |B|^{-\d (\eps)})$.\\
    
    Further let $\delta \in (0,2]$, 
    $N\ge 1$ be a sufficiently large integer, $N\le p^{c\delta}$ for an absolute constant $c>0$, 
    $\mathcal{A},\mathcal{B} \subseteq \F^2_p$ be sets and $g\in \SL_2 (\F_p)$ be a non--linear map. 
    Suppose that $\mathcal{S}$ is a set, 
    $\mathcal{S} \subseteq [N]^2 \times [N]^2$, $|\mathcal{S}| \ge N^{2+\delta}$.
    Then there is 
    a
    constant $\kappa = \kappa (\d) >0$ such that 
\begin{equation}\label{f:Z[i]_3}
    |\{ g(\a+a) = \beta+b ~:~ (\a,\beta) \in \mathcal{S},\,a\in \mathcal{A},\, b\in \mathcal{B} \}|
    - \frac{|\mathcal{S}||\mathcal{A}||\mathcal{B}|}{p}
    \ll_g \sqrt{|\mathcal{A}||\mathcal{B}|} |\mathcal{S}|^{1-\kappa} \,.    
\end{equation}
\label{t:Z[i]}
\end{theorem}
\begin{proof} 
    As we said before any $z\in \F_q$ can be written as $z=x+iy$ with $x,y\in \F_p$ and thus we obtain four sets $A,B,C,D \subseteq \F_q$, which correspond to $\mathcal{A}, \mathcal{B}, \mathcal{C}, \mathcal{D} \subseteq \F^2_p$. Let $\lambda = \a +i\beta \in \F_q$.
    Then it is easy to check that the equation
\[
    (a+b)(c+d) = \lambda \,, \quad \quad 
    a\in A,\, b\in B,\, c\in C,\, d\in D
\]
    coincides with \eqref{f:Z[i]_1}.
    Again, we can use the same argument 
    as in \cite{NG_S}, \cite{s_Kloosterman}, \cite{s_Chevalley} to obtain asymptotic formula \eqref{f:Z[i]_as} for the system \eqref{f:Z[i]_1}. The only thing we need to check that an analogue of the Helfgott growth result \cite{H} takes place for $\SL_2 (\F_q)$ (it is well--known that $\SL_2 (\F_q)$ is a quasi--random group for any $q \le p^{O(1)}$) but it was proved in \cite{Dinai}. 
    Further to obtain \eqref{f:Z[i]_2} we apply \eqref{f:Z[i]_1} with $\mathcal{A}= \mathcal{C} = (A+B)\times (A+B)$ and $\mathcal{B}=\mathcal{D} = - B\times B$.
    Then each solution of \eqref{f:Z[i]_1} is counted with the weight $|B|^4$ and the result follows.

    Finally, to get \eqref{f:Z[i]_3} we apply the same arguments as in the proof of Theorem \ref{t:BG_new}, using the ping--pong Lemma \ref{l:ping-pong_g} for $\Z[i]$ instead of $\Z$. Obviously, to calculate the norm of matrices we should use absolute values in $\C$ but not in $\mathbb{R}$. 
This completes the proof.
$\hfill\Box$
\end{proof}

\begin{remark}
    For simplicity, we have considered in \eqref{f:Z[i]_1}, \eqref{f:Z[i]_2} the transformation $gx = -1/x$. Of course, the same result takes place in general case (now $g \in \SL_2 (\F_q)$ is an arbitrary  non--linear transform) as in Theorems \ref{t:BG'}, \ref{t:BG_new}. 
\end{remark}

\section{On intersection of additive shifts of multiplicative subgroups}
\label{sec:shifts}

We begin with 
deriving some further consequences of Theorem \ref{t:BG'_intr}. 
Three parts of Theorem  \ref{t:shifts} below have the same spirit but there are some variations in parameters: we can take one or several shifts, we can control our shifts or not and, finally, there are several 
upper bounds for 
considered 
intersections of different quality. 
Also, let us remark 
that  in particular, inequality \eqref{f:shifts0} (with $n=1$) shows that for any $s\neq 0$ the set $\{1,u_{\pm s}, u^g_{\pm s}\}$ forms an expander in $\SL_2 (\F_p)$. 
Of course, Theorem 1 of \cite{BG} says that for any subset $S\subset \SL_2 (\Z)$ the Cayley graph $\Cay(\SL_2(\F_p), S)$ is an expander iff the subgroup $\langle S \rangle$ is non--elementary (that is,  $\langle S \rangle$  does not contain a solvable subgroup of finite index) and thus we can expect some properties of expansions of the set $\{1,u_{\pm s}, u^g_{\pm s}\}$. Nevertheless, our Theorem \ref{t:shifts} is a more delicate result, e.g., in the first part of this theorem the constant $c(\kappa)>0$ does not depend on $s$ (as \cite[Theorem 1]{BG} guarantees).

%
%
%
Recall that we write $a^g$ for $g^{-1}ag$ and let by definition $a^0:=a$. 
Below in this Section let $\kappa>0$ be the absolute constant 
from Theorem \ref{t:BG'_intr} or from  Theorem 
\ref{t:BG_new_intr}.

\begin{theorem}
    Let $g\in \SL_2 (\F_p)$ be a non--linear map. 
    Also, let $\kappa>0$ be the absolute constant from 
    Theorem \ref{t:BG'_intr}. 
    Then for any 
    $X\subseteq \F_p$,  and an integer parameter $N\ge 1$ the following holds:\\ 
$1)~$ 
    Suppose that 
$|X| \le p N^{-\kappa}$. 
    Then there is $s\in 2\cdot [N]$ such that 
\begin{equation}\label{f:shifts3}
    |(X+s) \cap g(X+s)| \le |X| \cdot N^{-\kappa} \,. 
\end{equation}
    In particular,  for any positive integer $n$ there are $s_j\in 2\cdot [N]$, $j\in [n]$ with
\begin{equation}\label{f:shifts0+}
    \left| \bigcap_{\eps_j \in \{ 0,1\}} (g^{u_{s_n}})^{\eps_n} \dots (g^{u_{s_1}})^{\eps_1} X \right|
    \le  |X| \cdot N^{-\kappa n} \,.
\end{equation}
$2)~$ If $Y:=g(X)$, 
$\eps\in (0,1/2]$ be any real number and  $|X| \le p N^{-\kappa}$, 
where 
$N\gg \eps^{1/\kappa}$,
then either 
$$
    |X\cap (X+2) \cap \dots \cap (X+2N)| \le \left(\frac12+\eps\right) |X| 
        \quad 
        \mbox{ or }
        \quad 
    |Y\cap (Y+2) \cap \dots \cap (Y+2N)| \le \left(\frac12+\eps\right) |Y| 
\,.$$
$3)~$ As above put  $Y:=g(X)$. If $|X|\le 3p/4$, then  for any $s\neq 0$  either $|X\cap (X+s)| \le (1-c(\kappa)) |X|$ or $|Y\cap (Y+s^{-1})| \le (1-c(\kappa)) |Y|$.
\\
In particular, for any positive integer $n$ the following holds 
\begin{equation}\label{f:shifts0}
    \left| \bigcap_{\eps_j \in \{ 0,1\}} u^{\eps_n g}_{s^{1-2\eps_n}_n} \dots u^{\eps_1 g}_{s^{1-2\eps_1}_1} X \right| \le (1-c(\kappa))^n |X| \,.
\end{equation}
\label{t:shifts} 
\end{theorem}
\begin{proof} 
    Let us begin with $1)$ because it is just a direct application of Theorem \ref{t:BG'_intr} 
    (thanks to this result  we can even assume that $g(x)=1/x$, say). 
    Indeed,  by formula \eqref{f:BG'_intr}, our assumption $|X| \le p N^{-\kappa}$  and the Dirichlet principle we can find $s\in 2\cdot [N]$ such that 
\[
    |(X+s) \cap g(X+s)| \ll \frac{|X|^2}{p} + |X| N^{-\kappa}
    \ll |X| N^{-\kappa} 
\]
as required. Since $|(X+s) \cap g(X+s)| = |X\cap g^{u_s} X|$,  we 
get 
\eqref{f:shifts0+} via iteration. 

    To obtain $2)$ let us consider the sets $X_{*} := X\cap (X+2) \cap \dots \cap (X+2N) \subseteq X$,
    $Y_{*} := Y\cap (Y+2) \cap \dots \cap (Y+2N) \subseteq Y$ and assume that $|X_*|\ge (1/2+\eps)|X|$, $|Y_*|\ge (1/2+\eps)|Y|$.
    Since $Y=g(X)$ and $X_*-2j \subseteq X$, $Y_*-2j \subseteq Y$, $\forall j\in [N]$, it follows that $|(Y_*-2j) \cap g(X_*-2j)|\ge 2\eps |X|$ for any  $j\in [N]$.  
    Hence the equation $y_*-2j = g(x_*-2j)$, where $j\in [N]$ and $x_* \in X_*$, $y_* \in Y_*$ has at least  $2\eps N |X|$ solutions. 
    Again this contradicts formula \eqref{f:BG'_intr} of Theorem \ref{t:BG'_intr}.

    Finally, to get $3)$ we use a variation of the argument from  \cite{He} and \cite{s_CDG}.  
    Let $|X\cap (X+s_1)| > (1-c) |X|$ and $|Y\cap (Y+s_2)| > (1-c) |Y|$ for some $s_1,s_2 \neq 0$.
    Dividing and redefining the sets $\tilde{X} = X/s_1$, $\tilde{Y}=Y/s_2$ we can assume that $s_1=s_2=1$ and let $\tilde{X} =\bigsqcup_{j\in J} I_j$, where $I_j$ are some intervals with step one (and similar to the set $\tilde{Y}$).
    Write $c=c(\kappa)$ for a sufficiently small constant, which we will choose later. 
    From 
    $|\tilde{X} \cap (\tilde{X}+1)| > (1-c) |X|$, 
    $|\tilde{Y}\cap (\tilde{Y}+1)| > (1-c) |\tilde{Y}|$, we see that $|J| \le  c|X|$. 
    Put $L=|X|/|J|$ and 
let $\omega \in (0,1)$ be a small parameter, which we will choose later.
One has  $\sum_{j\in J} |I_j| =|X|$ and hence $\sum_{j ~:~ |I_j|\ge \omega L} |I_j| \ge (1-\omega) |X|$. 
Splitting $I_j$ into intervals of length exactly $L_\omega := \omega L/2$, we see that the rest is at most 
$2\omega |X|$. 
Hence we have obtained some intervals $I'_i$, $i\in I$, having lengths  $L_\omega$ and step one and such that  $\sum_{i\in I} |I'_i| \ge (1-2\omega) |X|$.
Put $X' = \bigsqcup_{i\in I} I'_j$.
Similarly, we construct $Y' \subseteq Y$, $|Y'|\ge (1-2\omega)|Y|$. 
Then 
\begin{equation}\label{tmp:15.08_2}
    |X| = |Y| = |Y\cap g(X)| = |s_2 \tilde{Y} \cap g(s_1 \tilde{X})| \le 
    |s_2 Y' \cap g(s_1 X')|
    + 4 \omega |X| \,.
\end{equation}
    It follows that  $|s_2 Y' \cap g(s_1 X')| \ge (1-4\omega) |X|$. 
    Now our task is to obtain a good upper bound for the intersection. 
Let $\mathcal{M} = [\zeta L_\omega]$, where $\zeta = 2^{-6}$. 
Thus $M:=|\mathcal{M}| \ge \zeta L_\omega/4$.
We have $|(I'_i + m) \cap I'_i| \ge (1-2\zeta) |I'_i|$ for all $m\in \mathcal{M}$ and hence $|(X'+m) \cap X'| \ge (1-2\zeta)|X'|$.
Again, similarly, we obtain $|(Y'+m) \cap Y'| \ge (1-2\zeta)|Y'|$.
Recalling  $|s_2 Y' \cap g(s_1 X')| \ge (1-4\omega) |X|$, we see 
\[
    |s_2 (Y'+m_1) \cap g(s_1 (X'+m_2))| \ge |s_2 Y' \cap g(s_1 X')| - 4\zeta |X'|
    \ge (1-4\omega - 4\zeta)|X'| \ge 7|X'|/8 \,,
\]
    where we have chosen $\omega =\zeta = 2^{-6}$.
From above, we get 
\[
\frac{7|X'|M^2}{8} 
\le 
|\{ s_1 (y'+m_1) = g(s_2 (x'+m_2)) ~:~ x' \in X',\, y'\in Y',\, m_1,m_2 \in \mathcal{M} \}| \,.
\]
Let $g=(ab|cd)$, $c\neq 0$ and $h_{s_1,s_2} =(as_2,b|s_1 s_2 c, ds_1)$. 
Then the last equation can be rewritten as $y'+m_1 = h_{s_1,s_2} (x'+m_2)$. 
Now using $s_1 s_2 \equiv 1 \pmod p$, we see that the lower left corner of $h_{s_1,s_2}$ is  $O(1)$. Thus we can apply Theorem \ref{t:BG'} with $S=[M]\times [M]$ (also, see Remark \ref{r:lambda}) and obtain 
\[
\frac{7|X'|M^2}{8} 
\le
\frac{M^2|X'|^2}{p} + C_* |X'| M^{2-\kappa} \,,
\]
where $C_* >0$ is  an absolute constant, 
and whence 
\begin{equation}\label{f:eq_J}
M^2|X'| \ll 
|X'| M^{2-\kappa} 
\,. 
\end{equation}
Here  we have used the assumption $|X'|\le |X| \le 3p/4$. 
Estimate \eqref{f:eq_J} give us a contradiction for sufficiently large $M \ge C^{1/\kappa}$, where $C>1$ is an absolute constant. 
Recall that $M\ge 2^{-8}L_\omega = 2^{-15} L$.
Since $|J| \le c|X|$ and $L=|X|/|J| \ge c^{-1}$, it follows that 
$M\ge 2^{-15} c^{-1}$. Taking $c=c(\kappa)>0$ to be a sufficiently small number, we satisfy our condition $M \ge C^{1/\kappa}$. 

To obtain \eqref{f:shifts0} we just apply 1) to derive $|X\cap u_s X| \le (1-c(\kappa))|X|$ or  
$|g(X)\cap u_{s^{-1}} g(X)| \le (1-c(\kappa))|X|$ for any $s\neq 0$. The second possibility  is equivalent to $|X\cap u^g_{s^{-1}} X| \le (1-c(\kappa))|X|$. In any case we have found a word $w \in \{ u_s, u^{g}_{s^{-1}} \}$
such that $|X\cap wX| \le (1-c(\kappa))|X|$. After that we can iterate the obtained bound.
This completes the proof.
$\hfill\Box$
\end{proof}

\begin{remark}
    Let us say a few words about  Theorem  \ref{t:shifts}. Of course in the third  part of the result 
    the restriction $|X|\le 3p/4$ can be replaced to $|X|\le (1-\eps)p$ and it will just change $c(\kappa)$ to a certain positive constant $c(\kappa,\eps) >0$. Also, it is not possible to have just one shift in the first  part of Theorem  \ref{t:shifts}, formula \eqref{f:shifts3}. In view of Lemma \ref{l:action_g} below  a counterexample is very simple. 
    Indeed, let $X=\G \cap (\G-1) - 2$ and $s=2$. 
    Then using Lemma \ref{l:action_g}, we obtain $X+2 = \G \cap (\G-1) = (X+2)^{-1}$. Finally, a similar example shows that  one cannot take $\kappa>1$ in formula \eqref{f:shifts3}.    
\end{remark}

\begin{remark}
    It is possible to consider general functions differ from  M\"obius transformation similarly 
    to paper 
    \cite{s_CDG} and obtain some analogues of Theorem  \ref{t:shifts}. 
    Nevertheless, it gives much weaker bounds for the considered intersections. 
    E.g., if a sufficiently small set $X\subset \F_p$ satisfies $X= 1 + 1/(X-1)$, then for any $\lambda\neq 0$ one has $|X\cap \lambda X| \le (1-O(\exp(-\log p/\log \log p))) |X|$, see \cite[Theorem 1]{s_CDG}.
\end{remark}

Now we apply the bounds above to multiplicative subgroups.

Let $\G$ be a multiplicative subgroup of $\F^*_p$.
Also, let $k$ be a positive integer. 
Given non--zero numbers $\a_1,\dots,\a_k \in \F^*_p$ and any $\beta_1,\dots,\beta_k \in \F_p$ 
consider a generalization of sets \eqref{f:R_s} 
\begin{equation}\label{def:G_ab}
    \G_{\bar{\a}; \bar{\beta}} = \G_{\a_1,\dots,\a_k; \beta_1,\dots,\beta_k} = (\a_1 \G +\beta_1) \cap \dots \cap (\a_k \G +\beta_k) \,,
\end{equation}
where $\bar{\a}=(\a_1,\dots,\a_k)$ and $\bar{\beta}=(\beta_1,\dots,\beta_k)$. 
In this Section we 
want to obtain new upper bounds for size of the sets $\G_{\bar{\a}; \bar{\beta}}$. 
The best current results 
on additive shifts of multiplicative subgroups are contained in \cite{sv}. They say, basically, that for any sufficiently large $\G \le \F_p^*$, $|\G| \le p^{1-c}$ one has $|\G_{\bar{\a}; \bar{\beta}}| \ll k |\G|^{1/2+\eps_k}$, where $\eps_k \to 0$ as $k\to \infty$. 
We want somehow to break the "square--root barrier"\, (i.e. $\sqrt{|\G|}$), which appears above.

\bigskip 

First of all,  let us show how a matrix $g\in {\rm GL}_2(\F_p)$ acts on sets of the form \eqref{def:G_ab}.

\begin{lemma}
    Let $g=(ab|cd) \in {\rm GL}_2(\F_p)$, $c\neq 0$, $k\ge 2$, $\a_1,\dots,\a_k \in \F^*_p$ and $\beta_1,\dots,\beta_k \in \F_p$. 
    Suppose that $\beta_1 =  g^{-1} (\infty) = -d/c$. 
    Then we have 
\begin{equation}\label{f:action_g} 
    g \G_{\a_1,\dots,\a_k; \beta_1,\dots,\beta_k} = 
    \G_{\gamma_1,\dots,\gamma_k; g(\infty), g(\beta_2),\dots,g(\beta_k)} \,,
\end{equation}
    where 
    $\gamma_1 = -\frac{\det (g)}{c^2\a_1}$ and 
    $\gamma_j = \frac{\a_j \det(g)}{\a_1 c (c\beta_j + d)}$, $j\in [k]\setminus \{1\}$.\\
\label{l:action_g}
\end{lemma}
\begin{proof} 
    Let $|\G|=t$ and $g^{-1} = (AB|CD) = (d (-b)|(-c)a) \cdot  \det^{-1}(g)$.
    Take any $x\in g\G_{\bar{\a}; \bar{\beta}}$. 
    Then we have $g^{-1} x \in \G_{\bar{\a}; \bar{\beta}}$ and this is equivalent to 
\[
    \left( \frac{Ax+B}{Cx+D} - \beta_j \right)^t = \a^t_j,\,
    \quad \quad j\in [k]
 \]
    or, in other words, 
\[
    (dx-b-\beta_j (-cx+a))^t= \a^t_j (-cx+a)^t,\,
    \quad \quad j\in [k]
\]
    and hence 
\begin{equation}\label{tmp:14.08_1}
    \left( x-\frac{a\beta_j +b}{c\beta_j + d} \right)^t
    =
    (x-g (\beta_j))^t =  \a^t_j (-cx+a)^t (c\beta_j + d)^{-t} \,.
\end{equation}
    Let us show that the right--hand side of the last formula does not depend on $x$.
    We know that $x\in g\G_{\bar{\a}; \bar{\beta}}$ and in particular, $x \in g \G_{\a_1;\beta_1}$.
    In view of our condition $\beta_1 = -d/c$, we get
\[
    -cx+a = -c \cdot \frac{a(\a_1 \o +\beta_1)+b}{c(\a_1 \o +\beta_1)+d} + a
    =
    -c \cdot \frac{a(\a_1 \o +\beta_1)+b}{c\a_1 \o} + a
    = - \frac{a\beta_1+b}{\a_1} \o^{-1} \,,
\]
    where $\o \in \G$. 
    Hence $(-cx+a)^t = (- \frac{a\beta_1+b}{\a_1})^t$ and thus \eqref{tmp:14.08_1} can be rewritten as  
\[
    (x-g (\beta_j))^t = \left(\frac{-\a_j (a\beta_1+b)}{\a_1 (c\beta_j + d)} \right)^t 
    = \left( \frac{\a_j \det(g)}{\a_1 c (c\beta_j + d)} \right)^t 
\]
    for all $j \in [k]\setminus \{1\}$. 
    In other words,  $x\in \G_{\gamma_j,g(\beta_j)}$, $j>1$
    and, similarly, 
    $$
        g\G_{\a_1,\beta_1} = \G_{(a\beta_1+b)/c\a_1; a/c} = \G_{(a\beta_1+b)/c\a_1; g(\infty)} = \G_{-\frac{\det (g)}{c^2\a_1}; g(\infty)} \,.
    $$
    Notice that all $\gamma_j\neq 0$ because $g\in {\rm GL}_2(\F_p)$ and hence $\beta_1 = -d/c \neq -b/a$. 
    Thus we have 
    obtained 
    the inclusion
    $$
        g \G_{\a_1,\dots,\a_k; \beta_1,\dots,\beta_k} \subseteq 
    \G_{\gamma_1,\dots,\gamma_k; g(\infty), g(\beta_2),\dots,g(\beta_k)} \,.
    $$
    To get \eqref{f:action_g} apply the 
    last 
    inclusion and derive
\begin{equation}\label{tmp:15.08_1}
    |\G_{\gamma_1,\dots,\gamma_k; g(\infty), g(\beta_2),\dots,g(\beta_k)}| =     |g^{-1} \G_{\gamma_1,\dots,\gamma_k; g(\infty), g(\beta_2),\dots,g(\beta_k)}|
    \le 
    |\G_{\gamma'_1,\dots,\gamma'_k; g^{-1} (\infty), \beta_2,\dots,\beta_k}| \,,
\end{equation}
    where $\gamma'_j$ are some numbers, which we will calculate later.
    Here we have used the fact that $g(\infty) = a/c=-D/C$ and $a/c \neq -B/A$. 
    It is easy to check that 
    $$
        g^{-1} (\infty) = \frac{A}{C} = - \frac{d}{c} = \beta_1 \,,
    $$
    and thanks to $-\frac{d}{c} = \beta_1$, we get 
    $$
         \gamma'_1 = \frac{A g(\infty) + B}{C \gamma_1} =
         \frac{da/c-b}{-(a\beta_1+b)/\a_1} = \a_1 \,. 
    $$
    Thus it remains to show that $\gamma'_j = \a_j$ for all $j\in [k] \setminus \{1\}$ and then inequality \eqref{tmp:15.08_1} implies \eqref{f:action_g} because then 
    $\G_{\gamma'_1,\dots,\gamma'_k; g^{-1} (\infty), \beta_2,\dots,\beta_k} = \G_{\a_1,\dots,\a_k; \beta_1,\dots,\beta_k}$. 
    Taking $j>1$, we have
\[
    \gamma'_j = \frac{-\gamma_j (Ag(\infty)+B)}{\gamma_1 (C g(\beta_j) + D)}
    = \frac{ad/c-b}{-c \frac{a\beta_j+b}{c\beta_j+d}+a} \cdot \frac{c\a_1}{a\beta_1+b} \cdot \left(\frac{\a_j(a\beta_1+b)}{\a_1 (c\beta_j+d)} \right)
    =
    \a_j 
\]
as required. 
$\hfill\Box$
\end{proof}

\bigskip

It is easy to see that Lemma \eqref{l:action_g} implies formula  \eqref{f:R_s_inverse} from the Introduction. Another example is the following (the map below was considered in \cite{s_CDG}, say): let $g(x)=x/(x-1)$ and $X_{s_1,\dots,s_k} := (\G+1) \cap (\G+s_1) \cap \dots \cap (\G+s_k)$. In this case, we see that if $1-s_j\in \G$, $j\in [k]$, then $g(X_{s_1,\dots,s_k}) = X_{g(s_1),\dots,g(s_k)}$.

\bigskip 

\bigskip

Now we are ready to break the square--root barrier for subgroups. 

First of all, let us make a general remark. 
Let $\G < \F_p^*$ be a multiplicative subgroup and $S = \{ s_1,\dots,s_k \} \subseteq \F_p^*$ 
    be any set.
For an arbitrary  $c\in \F_p$ let us write $g_c = (01|1 (-c))$.
Take any $x\in \G_S$ and consider the subgroup $\langle u_{-x}, g_x \rangle \subset \SL_2 (\Z)$. It is easy to see that it is a  non--elementary subgroup (recall that $x\in \G_S$ and hence $x\neq 0$) and whence by \cite[Theorem 1]{BG} one has $|\G_S \cap u_{-x} \G_S \cap g_x (\G_S)| \le (1-\kappa (x) ) |\G_S|$.   
    We have $\det(g_{x}) = -1$ and $g_{x}(\infty) = 0$.
    Put $\overline{S} = S \cup \{-x\}$. 
    Using Lemma  \ref{l:action_g}, we obtain $g_{x} \G_{\overline{S}} = \G_{g_{x} \overline{S}}$ because in the notation of the lemma $\gamma_1=1$ and $\gamma_j = -1/(-s_j-x) = 1/(x+s_j) \in \G$, $j \in [|S|]$ thanks to $x\in \G_S$.
    Thus we get 
\begin{equation}\label{fc:shifts_G}
    |\G_{S, S+x, g_x (S), -x}| \le (1-\kappa (x)) |\G_S| \,.
\end{equation}
Iterating estimate  \eqref{fc:shifts_G}, we obtain Theorem \ref{t:shifts_intr} with $|T| \le 3^n |S| + O_n(1)$ but,  first of all, we can do slightly better and, secondly, we remove the dependence on $x$ in an analogue of \eqref{fc:shifts_G}.

Given $x\in \F^*_p$ write $w_x$ for $u^{w}_{x^{-1}}$. 

\begin{theorem}
    Let $\G < \F_p^*$ be a multiplicative subgroup, $\G=-\G$ and $S = \{ s_1,\dots,s_k \} \subseteq \F_q^*$
    be any set.
    Then there are numbers $1\le k_1 < k_2 < \dots < k_l = n$ 
    such that for any  elements $x_1,\dots,x_n$ with 
\[
    x_1 \in \G_{S}\,,\quad x_2 \in \G_{\bigcup_{\eps_1\in \{0,1\}} u^{\eps_1}_{x_1} S},\, \dots \,,  x_{k_1} \in \G_{\bigcup_{\eps_j\in \{0,1\}} u^{\eps_{k_1-1}}_{x_{k_1-1}} \dots u^{\eps_1}_{x_1} S}
    \,, \quad 
\]
\[
    x_{k_1+1} \in \G_{\{x_{k_1}\} \bigcup \bigcup_{\eps_j, \eta_1 \in \{0,1\}} w^{\eta_1}_{x_{k_1}} u^{\eps_{k_1-1}}_{x_{k_1-1}} \dots u^{\eps_1}_{x_1} S} \,, 
    \dots \,, 
\]
\begin{equation}\label{cond:inclusions}
    x_n \in \G_{\{x_{k_1}, \dots, x_{k_l}\} \bigcup \bigcup_{\eps_j, \eta_i \in \{0,1\}}
    u^{\eps_{x_n}}_{x_n} \dots u^{\eps_{k_l+1}}_{x_{k_l+1}} w^{\eta_l}_{x_{k_l}} \dots w^{\eta_1}_{x_{k_1}} u^{\eps_{k_1-1}}_{x_{k_1-1}} \dots u^{\eps_1}_{x_1} S}
\end{equation}
    one has 
\begin{equation}\label{f:shifts1_G} 
|\G_{\bigcup_{\eps_j, \eta_i \in \{0,1\}}
    u^{\eps_{x_n}}_{x_n} \dots u^{\eps_{k_l+1}}_{x_{k_l+1}} w^{\eta_l}_{x_{k_l}} \dots w^{\eta_1}_{x_{k_1}} u^{\eps_{k_1-1}}_{x_{k_1-1}} \dots u^{\eps_1}_{x_1} S}|
    \le 
        (1-\kappa)^n |\G_{S}| \,.
\end{equation}
    provided the sets from \eqref{cond:inclusions} are non--empty.\\
    Now let 
    $N$ be a positive integer, 
    $|\G_S| \le p N^{-\kappa}$.
    Then there is a vector $\vec{\a} = \vec{\a}(S)$ and there are at least  $N-N^{1-\kappa/2}$ numbers  $z\in 2\cdot [N]$ such that
\begin{equation}\label{f:shifts2_G} 
    |\G_{\vec{\a};\, z,S,g^{u_z}(S \cup 2\cdot [N] \setminus \{z\})}| \le |\G_{S}| N^{-\kappa/2} \,.
\end{equation}
\label{t:shifts_G} 
\end{theorem}
\begin{proof} 
    We use the third part of Theorem   \ref{t:shifts} with $X=\G_{S}$ and $s=x_1 \in \G_S$. If $|\G_S \cap (\G_S - x_1)|\le (1-\kappa)|\G_S|$, then apply  Theorem   \ref{t:shifts}  with $X=\G_S \cap (\G_S - x_1) = \G_{S,S+x_1}$ and $s=x_2$. And so on. 
    Finally, we find $Z=\G_{\cup_{\eps_j\in \{0,1\}} u^{\eps_{k_1-1}}_{x_{k_1-1}} \dots u^{\eps_1}_{x_1} S} :=\G_{T}$ such that for a certain  $x_{k_1}\in \G_T$ one has $|Z\cap (Z-x_{k_1})|>(1-\kappa)|Z|$. 
    In view of Theorem   \ref{t:shifts},
    we obtain 
    $|Z\cap w_{x_{k_1}}Z|\le (1-\kappa)|Z|$. 
    We have $w_{x_{k_1}} = (1 0|x^{-1}_{k_1} 1)$ and hence 
    $\det(w_{x_{k_1}}) = 1$ and $g_{x_1}(\infty) = x_{k_1}$.
    Put $\overline{T} = T\cup \{x_1\}$.
    Using Lemma \ref{l:action_g} 
    we obtain $w_{x_{k_1}} \G_{\overline{T}} = \G_{w_{x_{k_1}} \overline{T}}$ because in the notation of the lemma $\gamma_1=-x^2_{k_1} \in -\G = \G$ and $\gamma_j = x^2_{k_1}/(s_j+x_{k_1}) \in \G$, $j \in [|T|]$ 
     thanks to  $x_{k_1}\in \G_T$.
    Hence $\G_{\overline{T}} \cap w_{x_{k_1}} \G_{\overline{T}} = \G_{\{x_{k_1}\} \bigcup \bigcup_{\eps_j, \eta_1 \in\{0,1\}}  w^{\eta_1}_{x_{k_1}} u^{\eps_{k_1-1}}_{x_{k_1-1}} \dots u^{\eps_1}_{x_1} S}$. After that we iterate our procedure and derive \eqref{f:shifts1_G}.


    Finally, to get \eqref{f:shifts2_G}  we basically use the first  part of Theorem  \ref{t:shifts} or, in other words, Theorem \ref{t:BG'_intr} with $X=\G_S$ and $gx = 1/x$.
    In view of estimate \eqref{f:BG'_intr} of Theorem \ref{t:BG'_intr} 
    and 
    the average argument there are at least  $N-N^{1-\kappa/2}$ numbers $z \in 2\cdot [N]$ such  that 
\[
    |\G_S \cap g^{u_{z}} \G_S| = |\G_{S+z} \cap \G_{(S+z)^{-1}}| \le |\G_S|\cdot N^{-\kappa/2} \,.
\]
    One can check that $g^{u_{z}} = (-z (1-z^2)|1 z)$ and hence $g(\infty) = -z$. 
    Applying Lemma \ref{l:action_g} to 
    $\G_{S,2\cdot [N]}$, we see that
    $g^{u_{z}} \G_{S \cup 2\cdot [N]} \subseteq  \G_{\vec{\a};\, z, g^{u_z} (S \cup 2\cdot [N] \setminus \{z\})}$ with a certain $\vec{\a}=\vec{\a}(S)$. 
This completes the proof.
$\hfill\Box$
\end{proof}

\begin{remark}
    To 
    specify the set 
    $T=T(S)$, $S\subseteq T$ more concretely such that $|\G_{\vec{\a}(S);T}| \le (1-\kappa)^n |\G_S|$  
    say, in \eqref{f:shifts2_G} 
    one can write in the left--hand side of \eqref{f:shifts2_G} the following set:\\ $\G_{2\cdot [N],\cup_{z\in 2\cdot [N]} g^{u_z}(S \cup 2\cdot [N])}$. 
\end{remark}

\section{Appendix}
\label{sec:appendix}

The main aim of this Section is to prove bound \eqref{f:Burgess}
using an almost purely combinatorial method. 

\begin{theorem}
    Let $p$ be a prime number.
    Then 
\[
    d(p) \ll p^{1/4} \log p \,.
\]
\label{t:main_int}
\end{theorem}

To obtain Theorem \ref{t:main_int} we need lower bound for size of ratio sets of two intervals in $\F_p$ and it is crucial for us do not loose any power of logarithms. 
Our argument follows \cite{Burgess2} and \cite{CG}.

\begin{lemma}
    Let $H,H_*,a$ be positive integers, $H_* \le H$, $a+H<p$ and  
    $16 H^2_* H < p$. 
    Then for all sufficiently large $H_*$ the following holds 
\[
    \left| \left\{ \frac{y}{a+x} \in \F_p ~:~ (x,y) \in [H] \times [H_*] \right\} \right| \gg H_* H \,.
\]
\label{l:coprime_ratio}
\end{lemma}
\begin{proof}
    We need to solve the equation 
\[
    y(a+x') \equiv y'(a+x) \pmod p \,,
    \quad \quad 
    x,x'\in [H],\, y,y'\in [H_*] \,.
\]
    Using the Dirichlet principle, we multiply the last expression by an appropriate   
    $\lambda \in \F^*_p$ such that $|\lambda| \le \D$ and  for $A:= a\lambda \pmod p$ one has $|A| \le p/\D$, where our parameter $\D = (p/H)^{1/2}$. 
    Then we obtain the following equation  in $\Z$ 
\begin{equation}\label{tmp:12.08_1}
    y(A+\lambda x') = y'(A+\lambda x)  
\end{equation}
    because by our assumption 
\[
    2 H_* (p/\D + \D H)  = 4 H_* \sqrt{pH} < p \,.
\]
    In \eqref{tmp:12.08_1} we can assume that $(A,\lambda)=1$.
    Let $S \subseteq [H] \times [H_*]$ be the set of all pairs $(x,y)$ such that $(A+\lambda x,y)=1$. 
    Hence identity \eqref{tmp:12.08_1} implies $x=x'$, $y=y'$, provided $(x,y), (x',y') \in S$ and our task is to show $|S| \gg H_* H$ because the last fact implies the lemma in view of the Cauchy--Schwarz inequality \eqref{f:energy_CS}. 
    We have
\[
    |S| = \sum_{x\in [H]}\, \sum_{y\in [H_*]}\,  \sum_{d|(y,A+\lambda x)} \mu (d)
    =\sum_{d\le H_*} \mu(d) \sum_{y\in [H_*],\, d|y}\,\, \sum_{x\in [H],\, d|(A+\lambda x)} 1
    = 
\]
\[
    =
    \sum_{d\le H_*,\, (d,\lambda)=1} \mu(d) \left( \frac{H_*}{d} + \theta'_d \right) \left( \frac{H}{d} + \theta''_d \right) \,,
\]
    where for all $d$ one has $|\theta'_d|\le 1$, $|\theta''_d|\le 1$.
    Here we have used the assumption that $(A,\lambda)=1$. 
    It follows that 
\[
    |S| = H_* H \sum_{1\le d,\, (d,\lambda)=1} \frac{\mu(d)}{d^2} + O(H \log H_*) 
    \ge 
    H_* H \left(2-\frac{\pi^2}{6} \right) + O(H \log H_*) 
    \gg H_* H 
\]
as required. 
$\hfill\Box$
\end{proof}

\bigskip 

Now we are ready to prove Theorem \ref{t:main_int}. 

\bigskip 

\begin{proof} 
Let $P_a = a + \{0,1,\dots,d(p)-1\} = a + P_0 \subseteq R$. 
Also, let $N \subseteq \F_p^*$ be the set of all quadratic non--residues. 
For any positive integer $m\le d(p)-1$ we use the notation $P^{1/m}_0 = \{0,1,\dots,[(d(p)-1)/m]\}$ and, similarly, $P^{1/m}_a = P^{1/m}_0 + a$. 
We have
\begin{equation}\label{f:inclusion_R1}
    \frac{P^{1/2}_0}{P^{1/2}_a} \subseteq R-1 
\end{equation}
because 
\[
    \frac{P^{1/2}_0}{P^{1/2}_a} + 1 \subseteq \frac{P^{1/2}_0+P^{1/2}_a }{P^{1/2}_a} \subseteq \frac{P^{1}_a}{P^{1/2}_a} \subseteq R \,.
\]
Let $k$ be an integer parameter, which we will chose in a moment. 
Since $j P^{1/2k}_0 \subseteq P^{1/2}_{0}$, $j\in [k]$,  it follows from \eqref{f:inclusion_R1} that 
\[
    \frac{P^{1/2k}_0}{P^{1/2}_a} \subseteq (R-1) \cap (2^{-1}\cdot R-2^{-1}) \cap \dots (k^{-1} \cdot R-k^{-1}) \,.
\]
    Using Lemma \ref{l:coprime_ratio} with $H_* = |P_0^{1/2k}|$, $H=|P_a^{1/2}|$ (we can freely assume that $d(p) \ll p^{1/3}$) and formula \eqref{f:R_s} (for residues/non--residues), we get
\[
    d^2 (p) k^{-1} \ll |P^{1/2k}_0| |P^{1/2}_a| \le 
    \frac{p}{2^{k}} + \theta k\sqrt{p}  \ll k\sqrt{p} \,,
\]
    provided $k\gg \log p$. 
    Choosing $k\sim \log p$, we obtain 
    $d(p) \ll p^{1/4} \log p$ as required. 
$\hfill\Box$
\end{proof}

\bigskip 
\noindent{I.D.~Shkredov\\
Steklov Mathematical Institute,\\
ul. Gubkina, 8, Moscow, Russia, 119991}
\\
and
\\
IITP RAS,  \\
Bolshoy Karetny per. 19, Moscow, Russia, 127994\\
and 
\\
MIPT, \\ 
Institutskii per. 9, Dolgoprudnii, Russia, 141701\\
{\tt ilya.shkredov@gmail.com}

\end{document}